\documentclass[a4paper,12pt]{article}

\usepackage{graphicx}
\usepackage{amsmath, amsfonts, amsthm, amssymb, enumerate, hyperref,
fancyhdr, fullpage}
\usepackage[all]{xy}

\newcommand{\ignore}[1]{}

\newtheorem{thm}{Theorem}
\newtheorem{lemma}[thm]{Lemma}
\newtheorem{prop}[thm]{Proposition}
\newtheorem{conj}[thm]{Conjecture}
\newtheorem{cor}[thm]{Corollary}

\title{Seymour's Second Neighborhood Conjecture for Subsets of Vertices}

\author{Tyler Seacrest\thanks{The University of Montana Western,
710 S Atlantic St,
Dillon, MT 59725,
United States} \\
\texttt{tyler.seacrest@umwestern.edu}}

\date{\today}

\begin{document}

\maketitle

\abstract{Seymour conjectured that every oriented simple graph contains a vertex whose second neighborhood is at least as large as its first.  In this note, we put forward a  conjecture that we prove is actually equivalent:  every oriented simple graph contains a subset of vertices $S$ whose second neighborhood is at least as large as its first. 

This subset perspective gives some insight into the original conjecture.  For example, if there is a counterexample to the second neighborhood conjecture with minimum degree $\delta$, then there exists a counterexample on at most ${\delta + 1 \choose 2}$ vertices.

Given a vertex $v$, let $d_1^+(v)$ and $d_2^+(v)$ be the size of its first and second neighborhoods respectively.  A digraph is $m$-free if there is no directed cycle on $m$ or fewer vertices.  Let $\lambda_m$ be the largest value such that every $m$-free graph contains a vertex $v$ with $d_2^+(v) \geq \lambda_m d_1^+(v)$.   The second neighborhood conjecture implies $\lambda_m = 1$ for all $m \geq 2$.   Liang and Xu provided lower bounds for all $\lambda_m$, and showed that $\lambda_m \to 1$ as $m \to \infty$.   We improve on Liang and Xu's bound for $m \geq 3$ using this subset perspective.


}

\section{Introduction}

Unless otherwise noted, all digraphs in this paper are oriented simple graphs, and thus do not contain loops or two-cycles.    We will use $V(D)$ to denote the set of vertices of a digraph $D$.

Given a digraph $D$ and vertices $u$ and $v$, let $d(u, v)$ be the length of the shortest directed path from $u$ to $v$.   For this note, we consider $d(v, v)$ not to be zero, but the length of the shortest cycle containing $v$.  Let $N_k^+(v)$, the set of $k$th out-neighbors, be all vertices $u$ such that $d(v, u) = k$, and note that these sets are disjoint for fixed $v$.   We will use $N_k^-(v)$ to refer to the set of $k$th in-neighbors, defined analogously to out-neighbors.  For fixed $v$, the $N_k^-(v)$ are disjoint, though they may intersect with the $N_k^+(v)$.   Also, since we consider $d(v, v) \neq 0$,  $N_k^+(v)$ contains $v$ for some $k > 0$.   Let $d_k^+(v) = |N_k^+(v)|$ and $d_k^-(v) = |N_k^-(v)|$.  If $d_1^+(v) \leq d_2^+(v)$, we will call $v$ a \emph{Seymour vertex}.  For a set of vertices $S$, let $N_k^+(S)$ be all vertices $u$ such that $\min_{s \in S} d(s, u) = k$, and note that $N_1^+(S)$, $N_2^+(S)$, etc. are all disjoint. Again, because we have defined $d(v, v) \neq 0$, it is possible $S$ intersects with $N_k^+(S)$.   Define $d_k^+(S) = |N_k^+(S)|$.  

Seymour made the following conjecture, which has become known as Seymour's Second Neighborhood Conjecture.

\begin{conj}[Seymour, see~\cite{DeanLatka95}]
\label{conj:original}
Every oriented simple graph contains a Seymour vertex.
\end{conj}
We will use SNC to refer to this conjecture throughout this note.

The SNC, along with related conjectures of Caccetta and H\"aggvist~\cite{CaccettaHaggkvist78} and Ho\'ang and Reed~\cite{HoangReed87}, have remained open for decades.  (See Sullivan~\cite{Sullivan06} for a nice summary of results and conjectures related to the Caccetta-H\"aggvist conjecture.)   In this note, we introduce a new, related conjecture.  
\begin{conj}
\label{conj:subset}
Every oriented simple graph $D$ contains a non-empty, proper subset of the vertices $S$, such that $d_1^+(S) \leq d_2^+(S)$.
\end{conj}
Note that Conjecture~\ref{conj:subset} is clearly implied by the SNC, since if there is a Seymour vertex $v$, then we can simply let $S = \{v\}$ and Conjecture~\ref{conj:subset} follows.  We prove Conjecture~\ref{conj:subset} is actually equivalent to the SNC.  This follows from a lemma we prove in Section~\ref{sec:lemma}.   There may be some hope that Conjecture~\ref{conj:subset} is easier to prove than the SNC:  for example, Conjecture~\ref{conj:subset} has an easy proof for regular graphs (see Proposition~\ref{prop:reg}), a case that has received much attention but has yet to yield a proof for the SNC.   

Since $N_1^+(S)$ is a set which when removed disconnects the graph, it is possible Conjecture~\ref{conj:subset} is related to the isoperimetric method of Hamidoune. Using the isoperimetric method,  Hamidoune~\cite{hamidoune08} proved the SNC for vertex-transitive graphs, and later Llad\'o~\cite{llado13} proved the SNC for $r$-out-regular graphs of connectivity $r-1$.

In attempt to make progress on the SNC, Chen, Shen, and Yuster~\cite{ChenShenYuster03} posed the following problem:  Find the largest $\lambda$ such that one could prove the existence of a vertex $v$ such that 
\begin{equation}
\label{eqn:app}
d_2^+(v) \geq \lambda d_1^+(v).
\end{equation} 
They proved this approximate form of the SNC  for $\lambda \approx 0.6573 \ldots$, where the exact value of $\lambda$ is the real root of the equation $2 x^3 + x^2 - 1 = 0$.  They also claimed that $\lambda \approx 0.67815 \ldots$ was achievable with similar methods.

A digraph is $m$-free if it has no directed cycles with length at most $m$.    One can then ask the Chen, Shen, and Yuster question in regards to this restricted set of digraphs.   Let $\lambda_m$ be the largest value such that every $m$-free digraph has a vertex $v$ where $d_2^+(v) \geq \lambda_m d_1^+(v)$.   The second neighborhood conjecture implies $\lambda_m = 1$ for all $m \geq 2$. Zhang and Zhou~\cite{ZhangZhou10} showed $\lambda_3 \geq 0.6751$.    Liang and Xu~\cite{LiangXu17} improved this and extended the result for all $m$, showing that $\lambda_m$ is greater than the only real root in the interval $(0, 1)$ of the polynomial
$$
2x^3 - (m-3)x^2 + (2m-4)x - (m-1).
$$
This implies $\lambda_3 \geq 0.6823 \ldots$, which improved the Zhang and Zhou result for $\lambda_3$.   The bound on $\lambda_4$ was $0.7007 \ldots$, and in general, $\lambda_m \to 1$ as $m \to \infty$.

%\begin{thm}
%\label{thm:m-free}
%An $m$-free digraph contains a vertex $v$ with $d_2^+(v) \geq \alpha_m d_1^+(v)$, where $\alpha_m$ is the unique positive real root of
%$$
%x^{m} + x^{m-1} = 1
%$$
%Note $\alpha_m$ is a lower bound on $\lambda_m$.
%\end{thm}

To improve upon Liang and Xu's bounds, this note proves the following:

\begin{thm}
\label{thm:m-free}
The unique positive real root of $x^m + x^{m-1} = 1$ is a lower bound on $\lambda_m$.
\end{thm}

In other words, an $m$-free digraph $D$ will have a vertex $v$ such that $d_2^+(v) \geq \lambda d_1^+(v)$ for any nonnegative $\lambda$ satisfying $\lambda^m + \lambda^{m-1} \leq 1$.

For $2$-free digraphs, Theorem~\ref{thm:m-free} gives the golden ratio of $\lambda_2 \geq .6180 \ldots$, which is not as good as the Chen, Shen, and Yuster result.  However, Theorem~\ref{thm:m-free} gives $\lambda_3 \geq .7548 \ldots$ and $\lambda_4 \geq 0.8191 \ldots$, which does improve upon the Liang and Xu result.   In fact, our result provides the best-known bound for all $m \geq 3$.  Note that the Liang and Xu result  asymptotically gives a lower bound of $1 - \sqrt{2} \frac{1}{\sqrt{m}} + o\left( \frac{1}{\sqrt{m}} \right)$, while our result asymptotically gives a lower bound of $1 - \ln(2) \frac{1}{m} + o\left( \frac{1}{m} \right)$.

\section{Main Lemma}
\label{sec:lemma}

We say $D$ is a \emph{$\lambda$-counterexample} (to the SNC) if $d_2^+(v) < \lambda d_1^+(v)$ for all vertices $D$.   We say $D$ is an \emph{edge-minimal} $\lambda$-counterexample if one cannot remove edges to create a smaller $\lambda$-counterexample.  We say $D$ is \emph{minimal} $\lambda$-counterexample if one cannot remove edges, vertices, or both to create a smaller counterexample.  We need to discuss $d_k^+(v)$ for different digraphs in this proof, so let $d_k^{+}(v, D)$ represent the number of $k$ out-neighbors of $v$ specifically in graph $D$.

The following lemma says, starting with a counterexample to the SNC, one can remove edges so that $d_2^+(S) < \lambda d_1^+(S)$ for all subsets of vertices $S$ where $N_1^+(S)$ is non-empty.   If $N_1^+(S)$ is empty and $S$ is a proper, non-empty subset of the vertices, then the graph is no longer strongly connected, and hence we can remove vertices to create an even smaller counter example.  By removing edges and vertices, we can create a strongly connected counterexample to the SNC where $d_2^+(S) < \lambda d_1^+(S)$ as long as $S$ is non-empty and and not equal to $V(D)$.     This show that Conjecture~\ref{conj:subset} is equivalent to Conjecture~\ref{conj:original}.

\begin{lemma}
\label{lemma:contract}
Let $D$ be a edge minimal $\lambda$-counterexample to the SNC, and let $S$ be any subset of the vertices of $D$ such that $N_1^+(S)$ is non-empty.  Then $d_2^+(S) < \lambda d_1^+(S)$.
\end{lemma}

\begin{proof}
Choose a subset of vertices $T \subset N_1^+(S)$ to be maximal such that $\lambda|T| > |N_1^+(T) \setminus S|$, or $T = \emptyset$ if no such $T$ exists.  If $T = N_1^+(S)$, then $\lambda d_2^+(S) = \lambda |T| > |N_1^+(T) \setminus S| = d_2^+(S)$  and we are done.  So assume $T \subsetneq N_1^+(S)$.  Set $T' = N_1^+(S) - T$, and note $T'$ is non-empty since $N_1^+(S)$ is non-empty, and $T \neq N_1^+(S)$. % Let $S'$ be $S \cup T \cup N_1^+(T)$.

Now create a new graph $D'$ equal to $D$ but with all edges from $S$ to $T'$ removed.  We claim that $D'$ is an $\lambda$-counterexample to the SNC, contradicting the minimality of $D$.  Suppose $D'$ is not an $\lambda$-counterexample, so it has $v$ such that $d_2^+(v, D') \geq \lambda d_1^+(v, D')$.  Since we only removed outgoing edges from vertices in $S$, $v$ must be in $S$.  

Let $A$ be the set of first out-neighbors of $v$ in $D$ that are not first or second out-neighbors of $v$ in $D'$.  Similarly, let $B$ be the set of first outneighbors of $v$ in $D$ that are second out-neighbors of $v$ in $D'$.   Using out-neighborhoods within $D$, let $$C = N_1^+(A \cup B) \setminus (S \cup T \cup N_1^+(T)),$$ and note that every vertex in $C$ is a second out-neighbor of $v$ in $D$ but not in $D'$.  See Figure~\ref{fig:lemma-diagram} for a diagram of some of these sets.  We  have that $v$ satisfies $d_2^{+}(v, D) < \lambda d_1^{+}(v, D)$ in $D$ and satisfies $d_2^{+}(v, D') \geq \lambda d_1^{+}(v, D')$ in $D'$,  we have that
\begin{align*}
\lambda d_1^{+}(v, D) & >  d_2^{+}(v, D) \\ 
\lambda (d_1^{+}(v, D) - d_1^{+}(v, D')) & > d_2^{+}(v, D) - d_2^{+}(v, D') \\
\lambda(|A| + |B|) & > |C| - |B|.
\end{align*}

\begin{figure}
\begin{center}
\includegraphics[scale=0.6]{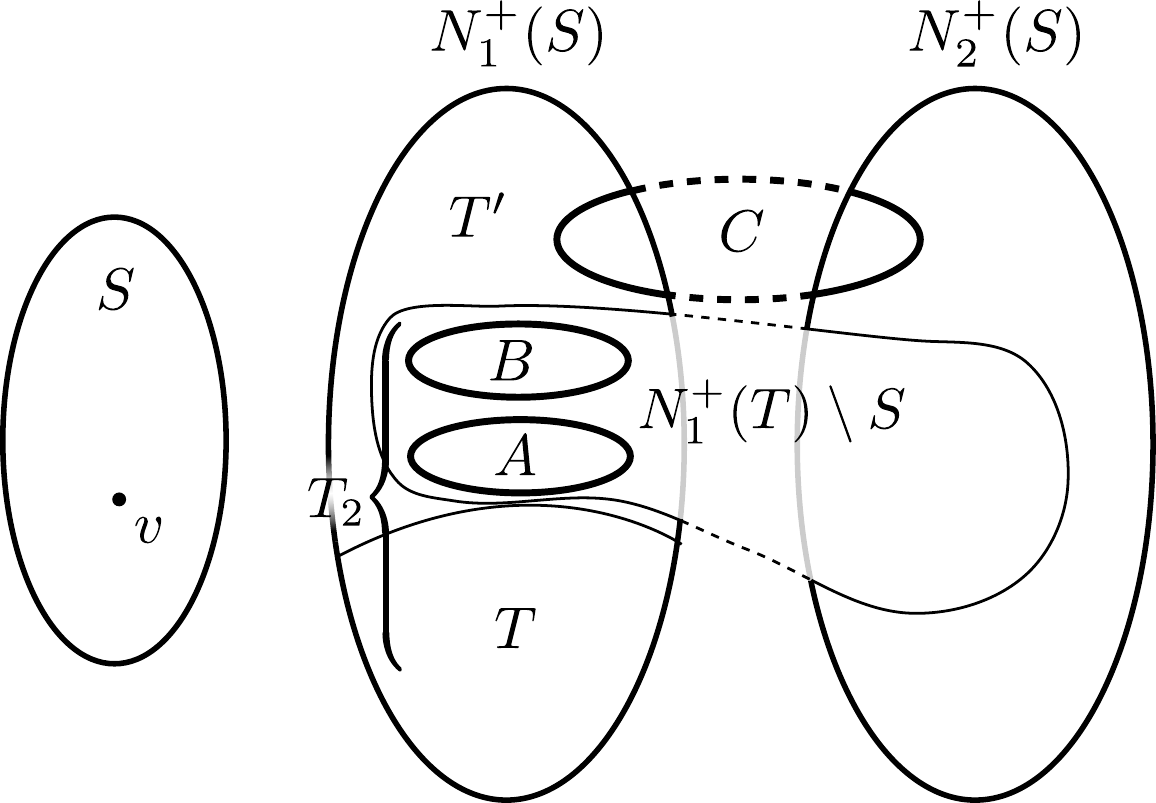}
\end{center}
\caption{\label{fig:lemma-diagram}  A diagram of some of the sets used in the proof of Lemma~\ref{lemma:contract}.}
\end{figure}

Set $T_2 = T \cup A \cup B$.   Since $B$ consists of second out-neighbors of $v$ in $D'$, but we removed all edges from $S$ to $B$, it must be the case that the vertices of $B$ are second out-neighbors of $v$ through $T$.  In other words, the vertices of $B$ lie inside $N_1^+(T)$.  Based on this fact about $B$ and how $C$ was defined, we have $|N_1^+(T_2) \setminus S| \leq |N_1^+(T) \setminus S| + |C| - |B|$.  By assumption, $|N_1^+(t) \setminus S| \leq \lambda |T|$, and we also have $|C| - |B| < \lambda(|A| + |B|)$.  Hence 
\begin{align*}
|N_1^+(T_2) \setminus S| & \leq |N_1^+(T) \setminus| + |C| - |B| \\
& <  \lambda(|T|) + \lambda(|A| + |B|) = \lambda|T_2|.
\end{align*}
But this contradicts the maximality of $T$.  
\end{proof}

\section{Quick Results}

%\begin{conj}
%\label{conj:further}
%Every digraph without loops or two-cycles contains a vertex $v$ such that $|N_j(v)| \leq |N_{j+1}|(v)|$ for some $j$.
%\end{conj}
%
%Note that both of these conjectures are equivalent to the SNC.   Now, if the SNC were false, then by Lemma~\ref{lemma:contract} we could assume every subset was such that $|N(S)| > N_2(S)|$, which would mean Conjecture~\ref{conj:subset} was also false.  Therefore, Conjecture~\ref{conj:subset} implies the SNC.   We also have that the SNC clearly implies Conjecture~\ref{conj:further}, and Conjecture~\ref{conj:further} implies Conjecture~\ref{conj:subset} by setting $S = \bigcup_{i = 1}^{j-1} N_i(v)$.    Therefore, all three conjectures are equivalent.

Lemma~\ref{lemma:contract} leads to a quick corollary using $\lambda = 1$.

\begin{cor}
If there exists a counterexample $D$ to the SNC with minimum degree $\delta$, then there exists a counterexample with at most ${\delta + 1 \choose 2}$ vertices.
\end{cor}

\begin{proof}
Let $D$ be a counterexample to the SNC with minimum degree $\delta$.  We can assume $D$ is edge-minimal, as that does not affect the number of vertices and at worse lowers the minimum degree.   Thus Lemma~\ref{lemma:contract} applies. 

Let $v$ be a vertex of minimum degree.  Using $S = \bigcup_{i = 1}^{k-1} N_i^+(v)$, we see Lemma~\ref{lemma:contract} gives $d_k^+(v) > d_{k+1}^+(v)$ for $k = 1, 2, 3, \ldots, \ell$, where $\ell$ is defined to be the first neighborhood of $v$ where $N_{\ell+1}^+(v)$ is empty.    Define  $U = \bigcup_{i = 1}^\ell N_i^+(v)$.  We see each neighborhood in this union is smaller than the last, so $U$ has at most
$$
\delta + (\delta - 1) + (\delta - 2) + \cdots + 1= {\delta + 1 \choose 2} 
$$
vertices.   While $U$ may not be all the vertices of the graph if $D$ is not strongly connected, it must contain a counterexample to the SNC, which gives the result.
\end{proof}

Kaneko and Locke~\cite{KanekoLocke01} showed that the SNC is true for all graphs with minimum degree at most $6$.   The previous corollary says that to extend this result to $7$, you could do so by showing there is no counterexample with $28$ vertices or fewer.  This is a tall order computationally, but a finite problem at least.

An \emph{in-regular} graph is a graph such that $|N_1^-(v)|$ is the same for all $v$.  Here we show that Conjecture~\ref{conj:subset} is true in the case of in-regular graphs.  Note that this proof unfortunately does not translate to the SNC since in-regular graphs are not closed under removal of edges and vertices, and therefore Lemma~\ref{lemma:contract} does not help. 

\begin{prop}
\label{prop:reg}
Given an in-regular digraph $D$ without loops or multiple edges, there exists a subset of vertices $S$ such that $d_1^+(S) \leq d_2^+(S)$.  
\end{prop}

\begin{proof}  % Add an abstract to this proof!  CHANGE!!
Consider a minimum, strongly-connected counterexample $D$ to this proposition.   Since $D$ would also be a counterexample to the SNC, for every vertex $v$, we have $d_1^+(v) > d_2^+(v)$.   Since $\sum_{v \in V(D)} d_1^+(v) = \sum_{v \in V(D)} d_1^-(v)$ and $\sum_{v \in V(D)} d_2^+(v) = \sum_{v \in V(D)} d_2^-(v)$, $d_1^+(v) > d_2^+(v)$ implies that $\sum_{v \in V(D)} d_1^-(v) > \sum_{v \in V(D)} d_2^-(v)$.  Therefore we know there exists at least one vertex $v$ such that $d_1^-(v) > d_2^-(v)$.     Let $\mathcal V$ be the set of all vertices such that $d_1^-(v) > d_2^-(v)$.

For every $v \in \mathcal V$, set $S_v = V(D) \setminus (N_1^-(v) \cup N_2^-(v))$.  Since $D$ is a counterexample, we know that $d_1^+(S_v) > d_2^+(S_v)$.  Notice that  
\begin{itemize}
\item $N_1^+(S_v) \subseteq N_2^-(v)$,
\item there are more vertices in $N_1^-(v)$ than $N_2^-(v)$, and
\item there are fewer vertices in $N_2^+(S_v)$ than $N_1^+(S_v)$.
\end{itemize}
Therefore, there are fewer vertices in $N_2^+(S_v)$ than in $N_1^-(v)$, so there must be some vertex $u$ in $N_1^-(v)$ not in $N_2^+(S_v)$.   This implies the first two in-neighborhoods of $u$ are contained in the first two in-neighborhoods of $v$.  Notationally, $\left(N_1^-(u) \cup N_2^-(u)\right) \subsetneq \left(N_1^-(v) \cup N_2^-(v)\right)$.    See Figure~\ref{fig:prop-diagram} for a diagram of some of these sets.

\begin{figure}
\begin{center}
\includegraphics[scale=0.6]{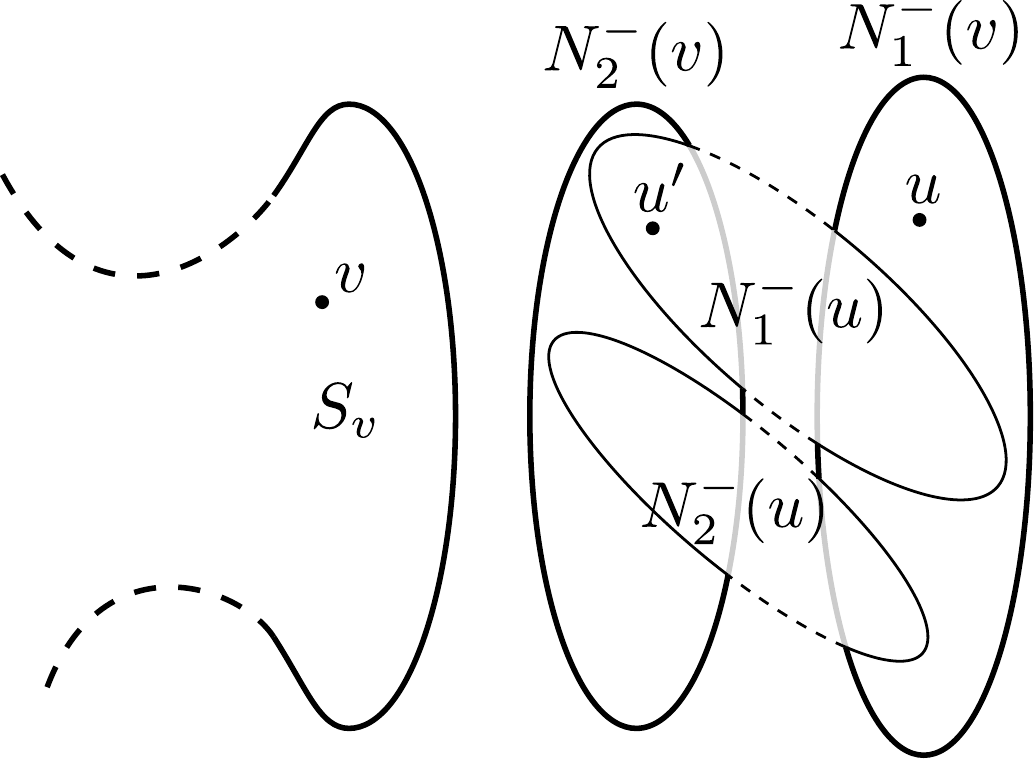}
\end{center}
\caption{\label{fig:prop-diagram}  A diagram of some of the sets used in the proof of Proposition~\ref{prop:reg}.}
\end{figure}

If $u \in \mathcal V$, then we can apply the same argument and get a $u'$ such that the first two in-neighborhoods of $u'$ are contained in the first two in-neighborhoods of $u$.  By repeating this argument, eventually we find a $u^*$ whose first two in-neighborhoods are contained in the first two  in-neighborhoods of $v$, but $u^* \notin \mathcal V$.    So $\left(N_1^-(u^*) \cup N_2^-(u^*)\right) \subsetneq \left(N_1^-(v) \cup  N_2^-(v)\right)$.  However, since $D$ is in-regular, we have $|N_1^-(u^*)| = |N_1^-(v)|$, and $|N_2^-(u^*)| \geq |N_2^-(v)|$, and so this containment is a contradiction. 
\end{proof}

\section{Approximate Second Neighborhood for $m$-free digraphs}

We know prove the main result, which is a restatement of Theorem~\ref{thm:m-free}.

\begin{thm}
\label{thm:reverse-radius}
Any $m$-free digraph $D$ has a vertex $v$ such that $d_2^+(v) \geq \lambda d_1^+(v)$ for $\lambda$ any real number between $0$ and $1$ satisfying
$$
\lambda^{m} + \lambda^{m-1} \leq 1.
$$
\end{thm}

\begin{proof}
We start with a rough outline of the proof.   We will assume $D$ is a minimal $\lambda$-counterexample to this theorem, and thus every vertex has a second neighborhood that is smaller (by a factor of $\lambda$) than the first neighborhood.   Applying Lemma~\ref{lemma:contract}, we see this implies that every subset of vertices has a smaller second neighborhood in this way.   This means that if we go forward from a vertex $v$, every neighborhood we look at is smaller than the last.   Using an averaging argument, we can find a vertex $w$ where moving backwards, these neighborhoods get  smaller as well.   Starting at $w$ and moving backwards (looking at $N_1^-(w)$, $N_2^-(w)$, $N_3^-(w)$, etc.) we find these neighborhoods get smaller and smaller, until we reach $N_{r-1}^-(w)$.   We then reverse directions and move forward, and these neighborhoods will be even smaller yet.   These neighborhoods moving forward will be so small that, even though there are potentially more of them, they can only cover all the vertices that were in-neighborhoods moving backwards $\left(N_1^-(w) \cup N_2^-(w) \cup N_3^-(w) \cup \cdots \right)$ if $\lambda$ is sufficiently big, which gives the result.  Refer to Figure~\ref{fig:thm-diagram} for an illustration of some of the sets that will be involved in the proof in the case $r = 4$.

\begin{figure}
\begin{center}
\includegraphics[scale=0.6]{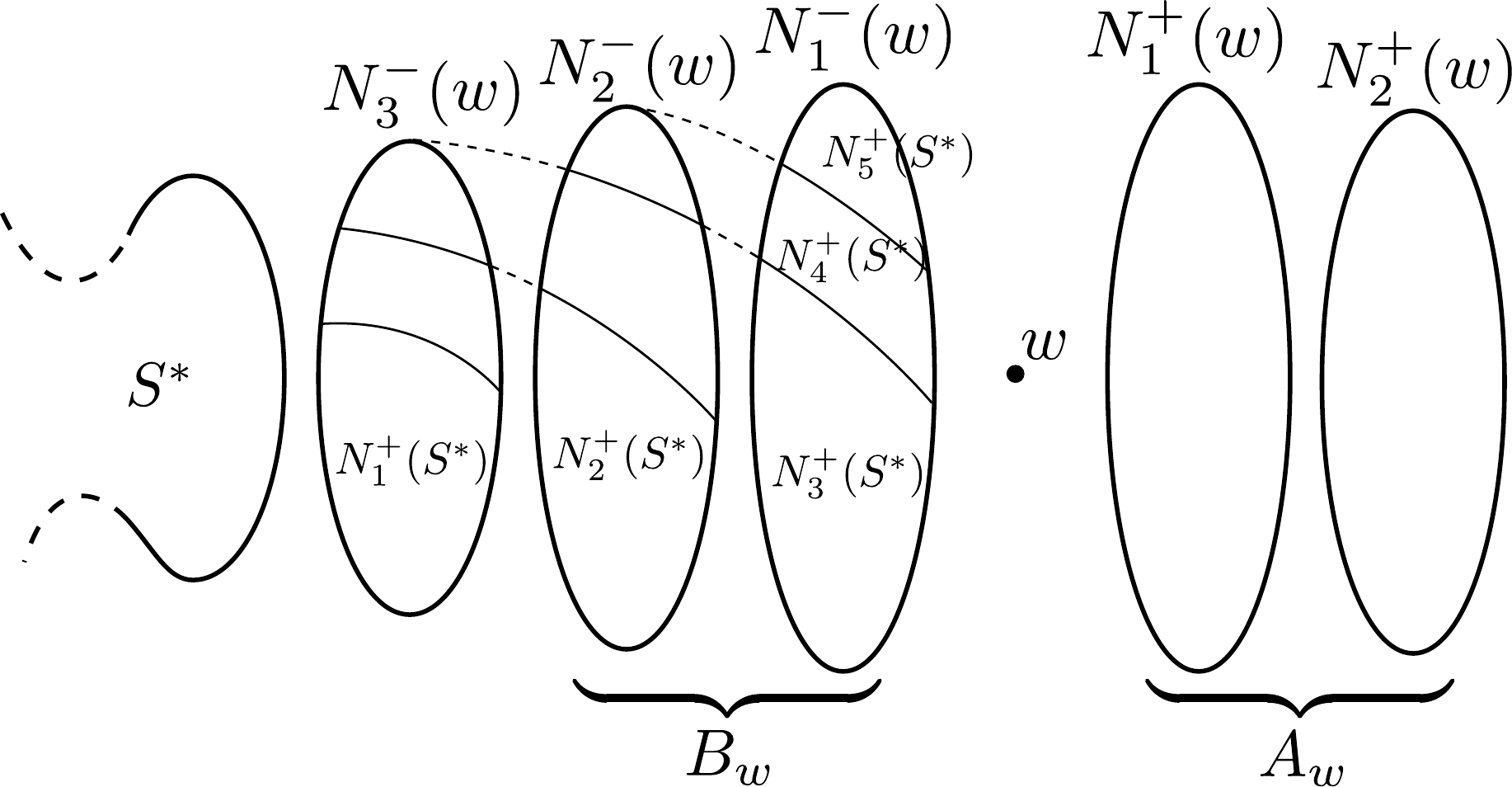}
\end{center}
\caption{\label{fig:thm-diagram}  A diagram of some of the sets for the case $r = 4$ used in the proof of Theorem~\ref{thm:reverse-radius}.}
\end{figure}

Consider a counterexample $D'$ to the statement;  that is, for a valid $\lambda$, $D'$ is $m$-free and satisfies $d_2^+(v) < \lambda d_1^+(v)$ for all vertices $v$.  Since it is a $\lambda$-counterexample to the SNC, it contains a strongly-connected minimal $\lambda$-counterexample $D$ to the SNC.   Since deleting edges and vertices cannot create a smaller cycle, $D$ is still $m$-free.    By Lemma~\ref{lemma:contract}, for every proper, non-empty subset of vertices $S$, $d_2^+(S) < \lambda d_1^+(S)$.  As we have seen, this implies that $d_{i+1}^+(v) < \lambda d_i^+(v)$ for all $i$ such that $d_i^+(v)$ is nonzero.  One can then show that this implies $d_{m}^+(v) < \lambda^{m-i}  d_{i}^+(v)$, and hence $d_i^+(v) >  \frac{1}{\lambda^{m-i}} d_{m}^+(v)$.

Let $A_v = \bigcup_{i = 1}^{m-1} N_i^+(v)$.    We see
\begin{align*}
|A_v| & =  \sum_{i = 1}^{m-1} d_i^+(v) \\
& > \sum_{i = 1}^{m-1} \frac{1}{\lambda^{r-1-i}} d_{r-1}^+(v) \\
& = \frac{1 - \lambda^{m-1}}{\lambda^{m-1}(1 - \lambda)} d_{m}^+(v)
\end{align*}
If we set $\gamma = \frac{\lambda^{m-1}(1 - \lambda)}{1 - \lambda^{m-1}}$, then we see that for every vertex $v$, $d_{m}^+(v) < \gamma |A_v|$.

Let $B_v = \bigcup_{i = 1}^{m-1} N_i^-(v)$.  Since $\sum_{v \in V} d_i^+(v) = \sum_{v \in V} d_i^-(v)$ for all $i$, we see that on average $B_v$ is the same size as $A_v$, and $d_{m}^+(v)$ is on average the same size as $d_{m}^-(v)$.  Therefore, since $d_{m}^+(v) < \gamma |A_v|$ for every vertex,  there must exist some vertex $w$ such that  $d_{m}^-(w) < \gamma |B_w|$.   

Note that since $D$ is strongly connected and $m$-free, for any $v$, $N_i^+(v)$ is non-empty for $i \leq m+1$.  This is because of how we defined $N_i^+(v)$, there is some $i$ such that $v \in N_i^+(v)$, and being $m$-free means this cannot happen until at least $i = m+1$.  Similarly, $N_i^-(v)$ is non-empty for $i \leq m+1$. 

Since $N_{m+1}^-(w)$ is non-empty, we set $S^* = \bigcup_{i = m+1}^\infty N_i^-(w)$.  Notice that we take an infinite union simply  because we want to keep going as long as the in-neighborhoods of $w$ are non-empty.   Since $N_1^+(S^*) \subseteq N_{m}^-(w)$, we have $|N_1^+(S^*)| \leq d_{m}^-(w) < \gamma |B_w|$.   By repeated use of $|N_2^+(S)| < \lambda |N_1^+(S)|$ for appropriate $S$, we see that $|N_2^+(S^*)| < \lambda \gamma |B_w|$, $|N_3^+(S^*)| < \lambda^{2} \gamma |B_w|$, etc., and in general, $|N_k^+(S^*)| < \lambda^{k-1} \gamma |B_w|$.   

Because this graph is strongly connected, these $N_k^+(S^*)$ must eventually cover $B_w$.   Therefore,
\begin{align*}
\sum_{i = 2}^\infty |N_i^+(S^*)| \geq |B_w| \\
\sum_{i = 2}^\infty \lambda^{i-1} \gamma |B_w| > |B_w| \\
\sum_{i = 2}^\infty \lambda^{i-1} \gamma  > 1 \\
\frac{\lambda}{1 - \lambda} \gamma > 1 \\
\frac{\lambda}{1 - \lambda} \cdot \frac{\lambda^{m-1}(1 - \lambda)}{1 - \lambda^{m-1}}  > 1 \\
\lambda^{m} > 1 - \lambda^{m-1}.
\end{align*}
This gives the result.
\end{proof}

\subsection*{Acknowledgments}

The author would like to thank Debbie Seacrest for her insights and valuable edits.

\bibliographystyle{plain}
\bibliography{2ndNeighborhood}

\end{document}